\documentclass[11pt]{article}

\usepackage{amsthm,amsmath,amssymb,amsfonts}
\usepackage{hyperref}
\usepackage{graphicx}
\usepackage{algorithmic}
\newtheorem{theorem}{Theorem}
\newtheorem{lemma}[theorem]{Lemma}
\newtheorem{proposition}[theorem]{Proposition}
\newtheorem{corollary}[theorem]{Corollary}
\newtheorem{example}[theorem]{Example}
\newtheorem{conjecture}[theorem]{Conjecture}
\newtheorem{question}[theorem]{Question}
\title{On Mubayi's Conjecture and conditionally intersecting sets}
\author{Adam Mammoliti, \thanks{
School of Mathematics and Statistics
UNSW Sydney
NSW 2052, Australia
}\\
\texttt{a.mammoliti@unsw.edu.au}
\and
Thomas Britz \footnotemark[1]\\
\texttt{britz@unsw.edu.au}
}

\begin{document}
\date{}
\maketitle
% REQUIRED
\begin{abstract}
Mubayi's Conjecture states
that if $\mathcal{F}$ is a family of $k$-sized subsets of $[n] = \{1,\ldots,n\}$
which, for $k \geq d \geq 3$ and $n \geq \frac{dk}{d-1}$, satisfies
$A_1 \cap\cdots\cap A_d \neq \emptyset$
whenever
$|A_1 \cup\cdots\cup A_d|~\leq~2k$
for all distinct sets $A_1,\ldots,A_d \in\mathcal{F}$,
then $|\mathcal{F}|\leq \binom{n-1}{k-1}$,
with equality occurring only if $\mathcal{F}$ is the family of all $k$-sized subsets
containing some fixed element.
This paper proves that Mubayi's Conjecture is true for all families
that are invariant with respect to shifting;
indeed, these families satisfy a stronger version of Mubayi's Conjecture.
Relevant to the conjecture,
we prove a fundamental bijective duality between
what we call $(i,j)$-unstable families and $(j,i)$-unstable families.
Generalising previous intersecting conditions,
we introduce the {\em $(d,s,t)$-conditionally intersecting} condition
for families of sets
and prove general results thereon.
We prove fundamental theorems on two $(d,s)$-conditionally intersecting
families that generalise previous intersecting families,
and we pose an extension of a previous conjecture by Frankl and F\"uredi.
Finally, we generalise a classical result by Erd\H{o}s, Ko and Rado
by proving tight upper bounds on the size of
$(2,s)$-conditionally intersecting families $\mathcal{F}\subseteq 2^{[n]}$
and by characterising the families that attain these bounds.
We extend this theorem
for sufficiently large $n$ to families $\mathcal{F}\subseteq 2^{[n]}$
whose members have at most a fixed size~$u$.
\end{abstract}



\section{Introduction}
\label{sec:introduction}
$ $\\
\noindent
Let $\mathcal{F}\subseteq\binom{[n]}{k}$ be a family of $k$-sized subsets of $[n]:=\{1,\ldots,n\}$.
The celebrated Erd\H{o}s-Ko-Rado~Theorem~\cite{MR0140419} states that if $n\geq 2k$ and
$A\cap B\neq \emptyset$ for all sets $A,B\in\mathcal{F}$,
then $|\mathcal{F}|\leq \binom{n-1}{k-1}$,
with equality when $n>2k$ occurring precisely when $\mathcal{F}$ is a {\em star},
i.e., when $\mathcal{F}$ is the family of all $k$-sized sets that contain a fixed element of $[n]$.
As Frankl~\cite{MR0398842} proved,
this theorem still holds when the intersection condition
$A\cap B\neq\emptyset$ is replaced by the more general condition that any $d$ sets of $\mathcal{F}$
have nonempty intersection, under the assumption that $n\geq dk/(d-1)$:
%\newpage
\begin{theorem}
\label{thm:MR0398842}
If $\mathcal{F}\subseteq \binom{[n]}{k}$ is a family of $k$-subsets of $[n]$
where $n\geq\frac{dk}{d-1}$ and $k\geq d\geq 3$
so that, for each $d$ sets $A_1,\ldots,A_d \in\mathcal{F}$,
\begin{equation}
\label{equ:MR0398842}
   A_1\cap\cdots\cap A_d \neq \emptyset\,,
\end{equation}
then
\[
  |\mathcal{F} | \leq \binom{n-1}{k-1} \,.
\]
Furthermore, equality holds if and only if $\mathcal{F}$ is a star.
\end{theorem}

According to Frankl and F\"{u}redi~\cite{MR729787},
Katona sought to extend the $d=3$ case of Frankl's result,
by relaxing the intersection condition~(\ref{equ:MR0398842})
to be required only when any $d=3$ distinct sets of~$\mathcal{F}$
have union containing at most $s$ elements for some $s\leq 3k$.
Frankl and F\"{u}redi~\cite{MR729787} proved that
this was possible when $2k\leq s\leq 3k$ and $n$ is sufficiently large.
Mubayi~\cite{MR2209708} completed this work by proving that
the $d=3$ case of Theorem~\ref{thm:MR0398842} could indeed be extended whenever $n\geq 3k/2$.
This led Mubayi~\cite{MR2209708} to conjecture a further
extension to all values of~$d$, as follows:

\begin{conjecture}\label{Conj2}{\rm \cite{MR2209708}}
If $\mathcal{F}\subseteq \binom{[n]}{k}$ is a family of $k$-subsets of $[n]$
where $n\geq\frac{dk}{d-1}$ and $k\geq d\geq 3$ so that,
for all distinct sets $A_1,\ldots,A_d \in\mathcal{F}$,
\begin{equation}
\label{equ:2k}
   A_1\cap\cdots\cap A_d \neq \emptyset
   \quad\textrm{whenever}\quad
  |A_1\cup \cdots\cup A_d| \leq 2k\,,
\end{equation}
then
\[
  |\mathcal{F} | \leq \binom{n-1}{k-1} \,.
\]
Furthermore, equality holds if and only if $\mathcal{F}$ is a star.
\end{conjecture}

As described above,
Mubayi~\cite{MR2209708} proved the conjecture for $d=3$,
and Mubayi~\cite{MR2355601} proved the conjecture for $d = 4$
when $n$ is sufficiently large.
F\"{u}redi and \"{O}zkahya~\cite{MR2591419}
and Mubayi and Ramadurai~\cite{MR2507945}
independently improved this result by proving
the conjecture for sufficiently large $n$,
thus generalising the above-mentioned result by Frankl and F\"{u}redi~\cite{MR729787}.
Chen et al.~\cite{MR2538649} proved Mubayi's Conjecture for $d=k$
and F\"{u}redi and \"{O}zkahya~\cite{MR2591419} proved that
Mubayi's Conjecture even holds when $d=k+1$.
However, Mubayi~\cite{MR2355601} provided a counterexample
that showed that the conjecture could not be extended to values of $d$ greater than or equal to $2^k$.
During the publication of this article,
Lifshitz~\cite{Lifshitz} also proved that
Mubayi's Conjecture holds
when $0< \zeta n \leq k \leq \frac{d-1}{d}n$ and $n$ is sufficiently larger than $\zeta$ and $d$,
and that, under these conditions,
the upper bound in Condition~(\ref{equ:2k})
can indeed be relaxed to $(\frac{d}{d-1}+\zeta)k$.

The first main result of the present paper,
Theorem~\ref{thm:Main} in Section~\ref{sec:main},
is a new partial verification of Mubayi's Conjecture.
Namely,
we prove that Mubayi's Conjecture holds for stable families~$\mathcal{F}\subseteq\binom{[n]}{k}$ of $k$-sets;
these are the families that are invariant with respect to the shifting operation.
Indeed, it turns out that Mubayi's Conjecture holds for such families
even when the upper bound in Condition~(\ref{equ:2k}) of the conjecture is relaxed.

Thus, to prove or disprove Mubayi's Conjecture,
it is sufficient to consider the conjecture with respect to families
that are not stable under shifting.
In Section~\ref{sec:dst-conditional}, we
prove general properties and characterisations of families that are unstable
and, more subtly, that are {\em $(i,j)$-unstable}.
The main result of that section is Theorem~\ref{thm:UW symmetry}
which describes how each $(i,j)$-unstable family
is related by an explicit and fundamental involution to a unique $(j,i)$-unstable family.

These notions of stability refer to the central concept of this paper,
namely the {\em $(d,s,t)$-conditionally intersecting} condition for families of sets,
introduced in Section~\ref{sec:dst-conditional}.
For $t=1$, this becomes the {\em $(d,s)$-conditionally intersecting} condition
which, in turn for $s = 2k$, is the intersecting condition~(\ref{equ:2k}).
More generally,
the $(d,s,t)$-conditionally intersecting condition
naturally generalises many previous intersecting conditions in
the literature
\cite{MR1429238,MR0140419,MR0398842,MR729787,MR905277,MR2507945};
those are here given a useful common framework.

In Section~\ref{sec:nonintersectingintersecting},
we pose a conjecture that sharpens Mubayi's Conjecture
by considering the difference between the two conditions in
Frankl's Theorem and Mubayi's Conjecture.
In particular, we conjecture on the size and extremal structures
of families $\mathcal{F}\in\binom{[n]}{k}$ that are $(d,2k)$-conditionally intersecting
but which are {\em not} intersecting;
see Conjecture~\ref{conj:nonintersecting}.
We show by example that this conjecture cannot be extended to small values of $n$
and we also show that there is circumstantial reason to believe that
Conjecture~\ref{conj:nonintersecting} might be true:
Proposition~\ref{prop:twinstar} shows that a similar claim is true
for $(3,4)$-conditionally intersecting families.

The families $\mathcal{F}\in\binom{[n]}{k}$ that are $(d,s)$-conditionally intersecting
for some $s<~2k$ include several intersection families in the literature~\cite{MR1405994,MR1429238,MR0140419,MR729787}.
In Section~\ref{sec:families},
we therefore define two $(d,s)$-conditionally intersecting
families that generalise these previous families.
We prove a fundamental theorem (Theorem~\ref{Thm Counter})
on these new families and thus also on previous families,
and we pose Conjecture~\ref{Conj Hr} which extends a previous
conjecture by Frankl and F\"uredi~\cite{MR729787}
on $(3,2k-1)$-conditionally intersecting families.

The final Sections~\ref{sec:2si} and~\ref{sec:2sii} are motivated by
the classical result by Erd\H{o}s et al.~\cite{MR0140419}
that an intersecting family $\mathcal{F}\subseteq 2^{[n]}$ can have size at most $2^{n-1}$
and that this bound is met by star families.
In Section~\ref{sec:2si}, we generalise this theorem by
proving tight upper bounds on the size of
$(2,s)$-conditionally intersecting families $\mathcal{F}\subseteq 2^{[n]}$
and by characterising the families that attain these bounds;
see Theorem~\ref{thm:mine}.
We extend these results further in Section~\ref{sec:2sii},
for certain parameters as well as for sufficiently large families
with respect to $(2,s)$-conditionally intersecting families $\mathcal{F}\subseteq 2^{[n]}$
whose members have at most a fixed number~$u$ members;
see Theorems~\ref{thm:mine2} and~\ref{thm:Family-size-comp}.
Finally, we invite the reader to ponder the question of whether
these final results can be merged into a complete description of
$(2,s)$-conditionally intersecting families $\mathcal{F}\subseteq 2^{[n]}$
and their extremal sizes and structures.

\section{Mubayi's Conjecture for stable families}
\label{sec:main}
$ $\\
\noindent
Initially used to prove the Erd\H{o}s-Ko-Rado~Theorem~\cite{MR0140419}
and surveyed in~\cite{MR905277,MR1105464},
the {\em $(i,j)$-shift}~$S_{ij}$
on each family $\mathcal{F} \subseteq 2^{[n]}$ is
the function defined by
\[
  S_{ij}(\mathcal{F}) := \{S_{ij}(A) \,:\,  A\in\mathcal{F}\}\,,
\]
where
\[
     S_{ij}(A)
  := \begin{cases}
       A' = (A - \{j\})\cup \{i\}  & \textrm{if } j\in A,\; i\notin A,\; A'\notin \mathcal{F}\,;\\
	   A                                 & \textrm{otherwise}.
     \end{cases}
\]
A few properties of the shift $S_{ij}$ are given below.
\begin{proposition}\label{prop:shift_prop}
If $\mathcal{F}\subseteq 2^{[n]}$ and $i,j\in[n]$,
then
\begin{itemize}
\item[(i)]   $|S_{ij}(A)|  =  |A|$\quad for all $A\in\mathcal{F}$;
\item[(ii)]  $|S_{ij}(\mathcal{F})| = |\mathcal{F}|$;
\item[(iii)]{for any $\mathcal{G} \subseteq \mathcal{F}$, $\begin{aligned}[t]
\displaystyle\Bigl|\bigcup_{A\in\mathcal{G}} A\Bigr|-1
                    &\leq \Bigl|\bigcup_{A\in\mathcal{G}} S_{ij}(A)\Bigr|
                    \leq \Bigl|\bigcup_{A\in\mathcal{G}} A\Bigr|+1
 \; \text{ and } \;\\ \displaystyle\Bigl|\bigcap_{A\in\mathcal{G}} A\Bigr|
                    &\leq \Bigl|\bigcap_{A\in\mathcal{G}} S_{ij}(A)\Bigr|
                    \leq \Bigl|\bigcap_{A\in\mathcal{G}} A\Bigr|+1.
\end{aligned}$}
\end{itemize}
\end{proposition}

\begin{proof}
(i) and (ii) follow from definitions (see also \cite{MR1105464}).
To prove (iii), note that the shift operation cannot increase or decrease
the size of a union or intersection of sets $A\in\mathcal{G}$ by more than one.
%Suppose that $\bigl|\bigcup_{A\in\mathcal{F}} S_{ij}(A)\bigr|
%                    = \bigl|\bigcup_{A\in\mathcal{F}} A\bigr| + 1$.
%Then the shift $S_{ij}$ has added $i$ to the union and yet not removed $j$,
%so some member $A\in\mathcal{F}$ must satisfy $(A-\{j\})\cup\{i\}\in\mathcal{F}$;
%however, then $i$ is in the unshifted union $\bigcup_{A\in\mathcal{F}} A$, a contradiction.
Suppose that $\bigl|\bigcap_{A\in\mathcal{G}} S_{ij}(A)\bigr|
                       = \bigl|\bigcap_{A\in\mathcal{G}} A\bigr| - 1$.
Then the shift $S_{ij}$ has removed $j$ from the intersection and yet did not add $i$,
so $i$ must already be contained in $\bigcap_{A\in\mathcal{G}} A$ and thus in each member $A\in\mathcal{G}$,
a contradiction.
\end{proof}

\noindent
%The inequalities in (3) above are tight:
%the family
%$\mathcal{F} := \bigl\{ \{1,3,4\}, \{2,3,5\}\}$ satisfies
%\[
%  \bigl|\bigcup_{A\in\mathcal{F}} A\bigr| - 1 = 4 = \bigl|\bigcup_{A\in\mathcal{F}} S_{12}(A)\bigr|\quad\text{and}\quad
%  \bigl|\bigcap_{A\in\mathcal{F}} S_{12}(A)\bigr| = 2 = \bigl|\bigcap_{A\in\mathcal{F}} A\bigr| + 1\,.
%\]
If $S_{ij}(\mathcal{F}) = \mathcal{F}$ whenever $1\leq i < j\leq n$,
then $\mathcal{F}$ is {\em stable} or {\em shifted}.
When $i<j$, the set $S_{ij}(A)$ either equals $A$
or replaces an element in~$A$ by an element of smaller value,
so each member $A$ of a stable family $\mathcal{F}$ must satisfy $S_{ij}(A) = A$.
Applying sufficiently many shifts $S_{i_\ell j_\ell}$ with $i_\ell<j_\ell$ to $\mathcal{F}$
will yield a (non-unique) stable family~$\mathcal{F}'$; see~\cite{MR1105464}.

\begin{example}
\label{exa:shiftinvariant}{\rm
The family $\mathcal{F}:=\bigl\{\{1,2\},\{1,3\}\bigr\}\subseteq 2^{[3]}$ is stable.
More generally, each star $\mathcal{F}\subseteq\binom{[n]}{k}$ whose members each contain the element~1 are stable.}
\end{example}

\noindent
The main result of the paper is Theorem~\ref{thm:Main} below
which states that Mubayi's Conjecture holds for stable families $\mathcal{F}$ of $k$-sets
even when Condition~(\ref{equ:2k}) of the conjecture is relaxed.

\begin{theorem}
\label{thm:Main}
If $\mathcal{F}\subseteq \binom{[n]}{k}$ is a stable family of $k$-subsets of $[n]$
where $n\geq\frac{dk}{d-1}$ and $k\geq d\geq 3$
so that, for all distinct sets $A_1,\ldots,A_d \in\mathcal{F}$,
\begin{equation}
\label{equ:2k-(d-3)}
   A_1\cap\cdots\cap A_d \neq \emptyset
   \;\;\textrm{whenever}\;\;
  |A_1\cup \cdots\cup A_d| \leq 2k-(d-2)\;,
\end{equation}
then
\[
  |\mathcal{F} | \leq \binom{n-1}{k-1} \,.
\]
Furthermore, equality holds if and only if $\mathcal{F}$ is a star.
\end{theorem}

\begin{proof}
Suppose that $\mathcal{F}$ is a family as given in the theorem.\\
If $n \leq 2k-(d-2)$,
then $|A_1\cup \cdots\cup A_d| \leq 2k-(d-2)$ is always true,
so Condition~(\ref{equ:2k-(d-3)}) simply requires
that any $d$ members of $\mathcal{F}$ intersect,
and the theorem follows from Theorem~\ref{thm:MR0398842}.

Suppose now that $n > 2k-(d-2)$ and
assume inductively that the theorem is true
for all integers $m = 2k-(d-2),\ldots,n-1$ and $d$ and $k$
for which $k \geq d\geq 3$ and $m\geq \frac{dk}{d-1}$.
Define
\[
  \mathcal{F}(\bar{n}) := \{ A               \::\: n\notin A\in\mathcal{F} \}
  \quad\textrm{and}\quad
  \mathcal{F}(n)       := \{ A\setminus\{n\} \::\: n   \in A\in\mathcal{F} \}\,.
\]
Since $\mathcal{F}(\bar{n})\subseteq \mathcal{F}$
and since $\mathcal{F}$ satisfies Condition~(\ref{equ:2k-(d-3)}),
the family $\mathcal{F}(\bar{n})$ must satisfy Condition~(\ref{equ:2k-(d-3)}). The family $\mathcal{F}(\bar{n})$ is clearly invariant under
any shift of the form $S_{in}$.
As $\mathcal{F}$ is stable, it follows that
$A' = (A - \{j\})\cup \{i\} \in \mathcal{F}$
for any $i<j<n$ and $A \in \mathcal{F}(\bar{n})$
such that $j \in A$ and $i \notin A$.
Furthermore, $A' \in \mathcal{F}(\bar{n})$ since $n \notin A'$.
Thus, $\mathcal{F}(\bar{n})$ is stable and, by assumption,
$|\mathcal{F}(\bar{n})|\leq \binom{n-2}{k-1}$.

Now consider $\mathcal{F}(n)$.
Let $S_{ij}$ be a shift with $i<j<n$, let $A\in \mathcal{F}(n)$, and set $B = A \cup \{n\}$.
Since $\mathcal{F}$ is stable, it follows that $S_{ij}(A) = S_{ij}(B)-\{n\} = B-\{n\} = A$,
so $\mathcal{F}(n)$ is stable.
Assume that $\mathcal{F}(n)$ does not satisfy Condition~(\ref{equ:2k-(d-3)}) for the value $d-1$.
That is,
assume that $A_1\cap\cdots\cap A_{d-1}=\emptyset$
for some members $A_1,\ldots,A_{d-1}\in\mathcal{F}(n)$ with
\begin{equation}
\label{equ:k-1,d-1}
  |A_1\cup\cdots\cup A_{d-1}|\leq 2(k-1)-((d-1)-2)\,.
\end{equation}
Set $F_j := A_j \cup \{n\}$ for each $j = 1,\ldots,d-1$.
By definition, $F_1 \cap\cdots\cap F_{d-1} = \{n\}$
and $F_j\in \mathcal{F}$.
Choose $i \in A_1 \cup\cdots\cup A_{d-1}$ and $A_\ell$ such that $i \notin A_\ell$,
and define the set $F'_\ell := (F_\ell - \{n\}) \cup \{i\}$.
Since $\mathcal{F}$ is stable and $i\notin F_\ell$ and $n\in F_\ell$,
it follows that $F'_\ell\in\mathcal{F}$.
Then by (\ref{equ:k-1,d-1}),
\[
       |F_1 \cup F_2 \cup\cdots\cup F_{d-1} \cup F'_\ell|
  \leq 2(k-1)-((d-1)-2)+1
   =   2k-(d-2)
\]
but
\[
  F_1 \cap\cdots\cap F_{d-1} \cap F'_\ell = \emptyset\,,
\]
a contradiction.
Hence, $\mathcal{F}(n)$ is stable and satisfies Condition~(\ref{equ:2k-(d-3)}) for the value $d-1$.
If $d > 3$, then $k-1 \geq d-1 \geq 3$ and,
as initially supposed,
\[
        n-1
   \geq 2k-(d-2)
    =  \frac{d-1}{d-2}(k-1) + \frac{(d-3)(k-d+2) + 1}{d-2}
    > \frac{d-1}{d-2}(k-1)
\]
so by the induction assumption, $|\mathcal{F}(n)|\leq \binom{n-2}{k-2}$.
If $d=3$, then $n-1 \geq 2k-(d-2) = 2k-1 > 2(k-1)$,
and since Condition~(\ref{equ:2k-(d-3)}) for $\mathcal{F}(n)$ reduces to
the condition of the Erd\H{o}s-Ko-Rado~Theorem for $d-1$,
it follows from that theorem that $|\mathcal{F}(n)|\leq \binom{n-2}{k-2}$.
Hence,
\[
  |\mathcal{F}|  =   |\mathcal{F}(\bar{n})| + |\mathcal{F}(n)|
       \leq \binom{n-2}{k-1} + \binom{n-2}{k-2}
        =   \binom{n-1}{k-1}\,.
\]
If $\mathcal{F}$ is a star, then equality is trivially attained above.
Conversely, suppose that $|\mathcal{F}| = \binom{n-1}{k-1}$.
Then, by the inequality above,
$|\mathcal{F}(\bar{n})| = \binom{n-2}{k-1}$ and $|\mathcal{F}(n)| = \binom{n-2}{k-2}$.
By the induction assumption,
$\mathcal{F}(\bar{n})$ forms a star
and if $d > 3$, then $\mathcal{F}(n)$ also forms a star.
By the Erd\H{o}s-Ko-Rado~Theorem,
$\mathcal{F}(n)$ also forms a star when $d = 3$.
Let $j$ denote the common element of the members of $\mathcal{F}(n)$,
and assume that $j\neq 1$.
Since $n-1 \geq k$, there exists $A \in \mathcal{F}(n)$ such that $1 \notin A$.
Thus, $S_{1j}(A) = (A-\{j\})\cup \{1\} \in \mathcal{F}(n)$ as $\mathcal{F}(n)$ is stable.
But $j\notin S_{1j}(A)$, a contradiction, so $j = 1$.
Similarly, the common element of the members of $\mathcal{F}(\bar{n})$ must be~1.
Hence, $\mathcal{F}$ is the star $\bigl\{A\,:\, 1\in A\in\binom{[n]}{k}\bigr\}$.
Induction concludes the proof.
\end{proof}

\section{A closer look at intersecting conditions}
\label{sec:dst-conditional}
$ $\\
\noindent
Theorem~\ref{thm:Main} verifies Mubayi's Conjecture for stable families
and indeed provides a stronger result for these families.
The following observation is thus worth highlighting.

\begin{corollary}
\label{cor:conjecture_reduction}
To prove or disprove Mubayi's Conjecture,
it is sufficient to consider the conjecture with respect to families that are not stable under shifting.
\end{corollary}

\noindent
Given this observation,
it is natural and perhaps important to investigate the general properties of families that are not stable;
that is the purpose of the present section.

A family of subsets $\mathcal{F}\subseteq 2^{[n]}$ is {\em $(d,s,t)$-conditionally intersecting}
if, for all distinct sets $A_1,\ldots,A_d \in\mathcal{F}$,
\begin{equation}
\label{equ:(d,s,t)}
  |A_1 \cap\cdots\cap A_d| \geq t
  \qquad\textrm{whenever}\qquad
  |A_1 \cup\cdots\cup A_d| \leq s  \,.
\end{equation}
If $\mathcal{F}$ is $(d,s,1)$-conditionally intersecting,
then $\mathcal{F}$ is {\em $(d,s)$-conditionally intersecting}.
Thus, the Erd\H{o}s-Ko-Rado~Theorem addresses families $\mathcal{F}$ that are $(2,n)$-conditionally intersecting.
Similarly, families $\mathcal{F}$ satisfying
Conditions~(\ref{equ:MR0398842}), (\ref{equ:2k}) and (\ref{equ:2k-(d-3)})
are, respectively,
$(d, n)$-conditionally intersecting,
$(d,2k)$-conditionally intersecting,
and
$(d,2k-(d-3))$-conditionally intersecting.
A family $\mathcal{F}$ is {\em $d$-wise $t$-intersecting} if it is $(d,n,t)$-conditionally intersecting;
that is, if
\begin{equation*}
  |A_1 \cap\cdots\cap A_d| \geq t
  \qquad\textrm{whenever}\qquad
   A_1,\ldots,A_d \in\mathcal{F}\,.
\end{equation*}
Families of $d$ sets $A_1,\ldots,A_d$ which do not $(d,2k)$-conditionally intersect
have been referred to in the literature as {\em $d$-clusters};
see \cite{MR2538649,MR2591419,MR2676834,MR2507945,MR2501436}.

The following proposition lists some observations regarding
$(d,s,t)$-conditionally intersecting families and shifting.

\begin{proposition}\label{prop:shift_and_intersectingness}
If $\mathcal{F}\subseteq \binom{[n]}{k}$ and $i,j\in[n]$,
then
\begin{itemize}
\item[(i)]   if   $\mathcal{F}$ is $(d,s,t)$-conditionally intersecting,
             then $\mathcal{F}$ is $(d,s-1,t)$-conditionally intersecting;
\item[(ii)]  if   $\mathcal{F}$ is $(d,s,t)$-conditionally intersecting and $s \geq dk$,
             then $\mathcal{F}$ is $d$-wise $t$-intersecting;
\item[(iii)] if   $\mathcal{F}$ is $d$-wise $t$-intersecting, %and $i<j$,
             then so is the shift $S_{ij}(\mathcal{F})$.
\end{itemize}
\end{proposition}

\begin{proof}
(i) and (ii) follow easily from definitions.
Proposition~\ref{prop:shift_prop} (iii) implies (iii) (see also \cite{MR1105464}).
\end{proof}

Shifting does not always preserve the
$(d,s,t)$-conditionally intersecting property.
The following proposition presents cases in which
shifting does preserve the $(d,s,t)$-conditionally intersecting property.

\begin{proposition}
\label{prop:Shift invariant unions}
If $A_1,\ldots, A_d \in 2^{[n]}$ are distinct
and either
\[
  |A_1 \cap\cdots\cap A_d|\geq t \qquad\textrm{or}\qquad
  |A_1 \cup\cdots\cup A_d|\geq s+2\,,
\]
then $\{S_{ij}(A_1),\ldots,S_{ij}(A_d)\}$ is $(d,s,t)$-conditionally intersecting.
\end{proposition}

\begin{proof}
If $|A_1 \cap\cdots\cap A_d|\geq t$,
then $\{A_1,\ldots,A_d\}$ is $d$-wise $t$-intersecting,
so by Proposition~\ref{prop:shift_and_intersectingness}~(iii),
the family $\{S_{ij}(A_1),\ldots,S_{ij}(A_d)\}$ is also $d$-wise $t$-intersecting
and thus $(d,s,t)$-condition\-ally intersecting.
Now suppose that $|A_1 \cup\cdots\cup A_d| \geq s+2$.
By Proposition~\ref{prop:shift_prop}~(iii),
$| S_{ij}(A_1) \cup\cdots\cup S_{ij}(A_d) | \geq s+1$,
so $\{S_{ij}(A_1),\ldots, S_{ij}(A_d)\}$ is trivially $(d,s,t)$-conditionally intersecting.
\end{proof}

Thus by this proposition,
if $A_1,\ldots, A_d\in 2^{[n]}$ are distinct
and the family \\$\{S_{ij}(A_1),\ldots , S_{ij}(A_d)\}$ is not $(d,s,t)$-conditionally intersecting,
then
\begin{equation}
\label{equ:notconditionallyintersecting}
  |A_1 \cap\cdots\cap A_d| \leq  t-1 \qquad\textrm{and}\qquad
  |A_1 \cup\cdots\cup A_d| \leq  s + 1\,.
\end{equation}
Hence in order to characterise when
the $(d,s,t)$-conditionally intersecting condition is not preserved under shifting,
we must consider families $\mathcal{F}\in\binom{[n]}{k}$
that are $(d,s,t)$-conditionally intersecting while $S_{ij}(\mathcal{F})$ is not, for some $i,j$.

A family $\mathcal{F}\subseteq 2^{[n]}$ with this property is {\em$(i,j)$-unstable} and must
contain a subfamily $\{A_1,\ldots,A_d\}\subseteq\mathcal{F}$ that is also $(i,j)$-unstable
and thus satisfies
\begin{equation}
\label{equ:unstable2}
  |A_1 \cap\cdots\cap A_d| \geq t \qquad\textrm{or}\qquad
  |A_1 \cup\cdots\cup A_d| \geq s+1
\end{equation}
as well as the inequalities (\ref{equ:notconditionallyintersecting}) and
\begin{equation}
\label{equ:unstable3}
  |S_{ij}(A_1) \cap\cdots\cap S_{ij}(A_d)| \leq t - 1 \qquad\textrm{and}\qquad
  |S_{ij}(A_1) \cup\cdots\cup S_{ij}(A_d)| \leq s\,.
\end{equation}
Thus by Proposition~\ref{prop:shift_prop},
$(i,j)$-unstable families may be characterised as follows.
\begin{proposition}
\label{prop:unstable}
A family of $d$ distinct sets $\mathcal{A} = \{A_1,\ldots,A_d\}\subseteq\mathcal{F}\subseteq 2^{[n]}$ is $(i,j)$-unstable
if and only if
\begin{align}
       |A_1 \cap\cdots\cap A_d|
  \leq |S_{ij}(A_1) \cap\cdots\cap S_{ij}(A_d)|\hspace*{6.6mm}
 &\leq  t - 1\hspace*{11mm}\textrm{and}\nonumber\\[.5mm]
       |A_1 \cup\cdots\cup A_d|
  =    |S_{ij}(A_1) \cup\cdots\cup S_{ij}(A_d)| + 1
 &=     s + 1\,.\label{equ:unstable}
\end{align}
\end{proposition}
\begin{proof}
If $\mathcal{A}$ is $(i,j)$-unstable,
then (\ref{equ:unstable}) follows from (\ref{equ:notconditionallyintersecting}), (\ref{equ:unstable2}) and (\ref{equ:unstable3}),
using Proposition \ref{prop:shift_prop}.
Conversely, if (\ref{equ:unstable}) is true,
then  $|A_1 \cup\cdots\cup A_d| > s$,
so $\{A_1,\ldots,A_d\}$ is $(d,s,t)$-conditionally intersecting
-- and, in contrast, $\{S_{ij}(A_1),\ldots,S_{ij}(A_d)\}$ is not:
 $|S_{ij}(A_1) \cap\cdots\cap S_{ij}(A_d)| < t$ and
 $|S_{ij}(A_1) \cup\cdots\cup S_{ij}(A_d)| \leq s$.
\end{proof}
\begin{example}
The family
$\mathcal{F} = \{ \{1,3 \} , \{2,4 \} ,\{3,5 \} \}$
is $(3,4,1)$-conditionally intersecting and is $(1,2)$-unstable since
the shift $S_{12}(\mathcal{F}) = \{ \{1,3 \} , \{1,4 \}, \{3,5 \} \}$
is not $(3,4,1)$-conditionally intersecting.
\end{example}

The following lemma describes an important property of
$(i,j)$-unstable subfamilies~$\mathcal{A}$ of a
$(d,s,t)$-conditionally intersecting family $\mathcal{F}$
that contain $d$ elements.
In particular, the sets in $\mathcal{A}$ containing $j$ but not $i$ can be shifted non-trivially.

\begin{lemma}\label{lem:UW-partition}
If $\mathcal{A} = \{A_1,\ldots,A_d\}\subseteq\mathcal{F} \subseteq 2^{[n]}$ is $(i,j)$-unstable,
then each set $A_\ell$ is contained in one of
\begin{align*}
  \mathcal{A}_{\bar{i}j}       &= \{A\in\mathcal{A} \,:\, j\in    A,\, i\notin A,\, A' = (A - \{j\}) \cup \{i\}\notin\mathcal{F}\}
                       = \{A\in\mathcal{A} \,:\, S_{ij}(A)\neq A\};\\
  \mathcal{A}_{i\bar{j}}       &= \{A\in\mathcal{A} \,:\, j\notin A,\, i\in    A\};\\
  \mathcal{A}_{\bar{i}\bar{j}} &= \{A\in\mathcal{A} \,:\, j\notin A,\, i\notin A\}\,.
\end{align*}
Furthermore,
$\mathcal{A}_{\bar{i}j}$ and $\mathcal{A}_{i\bar{j}}$ are nonempty,
and if $|A_1 \cap \cdots \cap A_d| = t - 1$, then $\mathcal{A}_{\bar{i}\bar{j}}$ is also nonempty.
\end{lemma}

\begin{proof}
%Suppose that $\mathcal{A}$ is $(i,j)$-unstable.
By Condition~(\ref{equ:unstable}),
$i,j  \in        A_1  \cup\cdots\cup        A_d$,
$j \notin S_{ij}(A_1) \cup\cdots\cup  S_{ij}(A_d)$, and
$i \notin        A_1  \cap\cdots\cap         A_d$.
If $j \in A_\ell$, then $S_{ij}(A_\ell) \neq A_\ell$,
so $i\notin A_\ell$ and $A_\ell' = (A_\ell - \{j\}) \cup \{i\} \notin \mathcal{F}$.
Hence, any set containing $j$ is contained in $\mathcal{A}_{\bar{i}j}$
and as $j\in A_1\cup\cdots\cup A_d$,
$\mathcal{A}_{\bar{i}j}$ is nonempty.
As each set containing $j$ does not contain $i$,
each of the sets $A_\ell$ containing $i$ must be elements of $\mathcal{A}_{i\bar{j}}$.
Also, $\mathcal{A}_{i\bar{j}}$ is non-empty since $i\in A_1\cup\cdots\cup A_d$.
If~$|A_1\cap\cdots\cap  A_d| = t-1$,
then $i \notin A_1\cap\cdots\cap  A_d = S_{ij}(A_1)\cap\cdots\cap S_{ij}(A_d)$,
so there exists $A_\ell$ such that $i\notin S_{ij}(A_\ell)$,
which means $i,j \notin A_\ell$
and so the set $\mathcal{A}_{\bar{i}\bar{j}}$ must be non-empty.
\end{proof}

The main result of this section, below,
shows how $(i,j)$-instability and $(j,i)$-instability form dual symmetries
that are connected by simple bijections.
In particular,
swapping the sets in $A\in \mathcal{A}_{i\bar{j}}$ by $(A - \{i\}) \cup \{j\}$
whenever $A$ is fixed by $S_{ji}$
defines a bijection between $(i,j)$-unstable families $\mathcal{A}$
and $(j,i)$-unstable families $\mathcal{B}$.

\begin{theorem}
\label{thm:UW symmetry}
Suppose that $\mathcal{A} = \{A_1,\ldots,A_d\}\subseteq\mathcal{F}\subseteq 2^{[n]}$ is $(i,j)$-unstable.
Let $\mathcal{A}_{\bar{i}j}$, $\mathcal{A}_{i\bar{j}}$ and $\mathcal{A}_{\bar{i}\bar{j}}$ be defined as in Lemma~\ref{lem:UW-partition},
suppose that $\mathcal{A}_{\bar{i}\bar{j}}\neq\emptyset$,
and define
\[
  \mathcal{G}  = \{A\in \mathcal{A}_{i\bar{j}} \,:\, S_{ji}(A) = A\}\qquad\text{and}\qquad
  \mathcal{G}' = \{(A - \{i\}) \cup \{j\} \,:\, A \in \mathcal{G}\}\,.
\]
Then $\mathcal{G} \neq \mathcal{A}_{i\bar{j}}$, and
\[
    \mathcal{B}
  = \mathcal{A}_{\bar{i}j}\cup \mathcal{A}_{\bar{i}\bar{j}}\cup (\mathcal{A}_{i\bar{j}} - \mathcal{G}) \cup \mathcal{G}'
\]
is a $(j,i)$-unstable family of $d$ sets $\{B_1,\ldots,B_d\}\subseteq\mathcal{F}$.

Conversely, define families $\mathcal{B}_{\bar{j}i}$, $\mathcal{B}_{j\bar{i}}$ and $\mathcal{B}_{\bar{i}\bar{j}}$
analogously for $\mathcal{B}$ as in Lemma~\ref{lem:UW-partition}.
Then
\[
    \mathcal{A}
  = \mathcal{B}_{\bar{j}i}\cup \mathcal{B}_{\bar{i}\bar{j}}\cup (\mathcal{B}_{j\bar{i}} - \mathcal{H}) \cup \mathcal{H}'
\]
where
\[
  \mathcal{H}  = \{B\in \mathcal{B}_{j\bar{i}} \,:\, S_{ij}(B) = B\}\qquad\text{and}\qquad
  \mathcal{H}' = \{(B - \{j\}) \cup \{i\} \,:\, B \in \mathcal{H}\}\,.
\]
\end{theorem}

\begin{proof}
To prove that $\mathcal{B}\subseteq \mathcal{F}$,
note that $\mathcal{A}_{\bar{i}j}, \mathcal{A}_{\bar{i}\bar{j}}, \mathcal{A}_{i\bar{j}}\subseteq \mathcal{A}\subseteq \mathcal{F}$
and  that if $B = (A - \{i\})\cup\{j\}\in \mathcal{G}'$ where $A\in \mathcal{G}$,
then $B \in \mathcal{F}$ by definition of $S_{ji}$
since $S_{ji}(A) = A$ and $A\in \mathcal{A}_{i\bar{j}}$.

Next, note that
\[
    |(\mathcal{A}_{i\bar{j}}-\mathcal{G})\cup \mathcal{G}'|
  = |\{ A\in \mathcal{A}_{i\bar{j}} \,:\, S_{ji}(A)\neq A\}|
  + |\{ A\in \mathcal{A}_{i\bar{j}} \,:\, S_{ji}(A) =   A\}|
  = |\mathcal{A}_{i\bar{j}}|\,,
\]
so by Lemma~\ref{lem:UW-partition},
\[
    |\mathcal{B}|
  = |\mathcal{A}_{\bar{i}j}| + |\mathcal{A}_{\bar{i}\bar{j}}| + |(\mathcal{A}_{i\bar{j}} - \mathcal{G}) \cup \mathcal{G}'|
  = |\mathcal{A}_{\bar{i}j}| + |\mathcal{A}_{\bar{i}\bar{j}}| + |\mathcal{A}_{i\bar{j}}|
  = |\mathcal{A}|
  = d\,.
\]

Let us now prove that $\mathcal{B}$ is $(j,i)$-unstable.
As $\mathcal{B} \subseteq \mathcal{F}$, $\mathcal{B}$ is $(d,s,t)$-conditionally intersecting.
Since $\mathcal{A}_{\bar{i}\bar{j}}\neq \emptyset$
and since the sets in $\mathcal{B}$ are equal to those in $\mathcal{A}$
or obtained from such by replacing $j$ with $i$,
it follows that $i,j\notin B_1\cap\cdots\cap B_d$.
Hence by Proposition~\ref{prop:unstable},
$|B_1 \cap\cdots\cap  B_d | = |A_1 \cap\cdots\cap  A_d | \leq t-1$.
Similarly since $\mathcal{A}_{\bar{i}\bar{j}}\subseteq \mathcal{B}$ is non-empty and unchanged by $S_{ji}$,
it follows that $i,j\notin S_{ji}(B_1) \cap\cdots\cap  S_{ji}(B_d)$,
so $|S_{ji}(B_1) \cap\cdots\cap  S_{ji}(B_d) | = |B_1 \cap\cdots\cap  B_d | \leq t-1$.

By Lemma~\ref{lem:UW-partition}, $\mathcal{A}_{\bar{i}j}$ is nonempty,
so $j$ is contained in
$A_1 \cup\cdots\cup  A_d$,
$B_1 \cup\cdots\cup  B_d$ and
$S_{ji}(B_1) \cup\cdots\cup  S_{ji}(B_d)$.
Thus, these unions differ only in whether they contain $i$ or not.
By definition of $\mathcal{G}$,
$S_{ji}(B) \neq B$ for every $B \in \mathcal{A}_{i\bar{j}} - \mathcal{G}$.
Thus, $i \notin S_{ji}(B_1) \cup\cdots\cup  S_{ji}(B_d)$
since every set in $\mathcal{B}$ containing $i$ is in $A_{i\bar{j}} - \mathcal{G}$.
By Lemma~\ref{lem:UW-partition},
$\mathcal{A}_{i\bar{j}}$ is nonempty
so $i\in A_1 \cup\cdots\cup  A_d$,
and by Proposition~\ref{prop:unstable},
$|A_1 \cup\cdots\cup  A_d| = s+1$.
Hence,
\[
  s   =  |S_{ji}(B_1) \cup\cdots\cup  S_{ji}(B_d)|
    \leq |B_1 \cup\cdots\cup  B_d|
    \leq |A_1 \cup\cdots\cup  A_d|
      =  s+1\,.
\]
In particular,
$|S_{ji}(B_1) \cap\cdots\cap  S_{ji}(B_d) | \leq t-1$ and
$|S_{ji}(B_1)\cup\cdots\cup  S_{ji}(B_d)| = s$,
so $S_{ji}(\mathcal{B})$ is not $(d,s,t)$-conditionally intersecting,
and it follows that $\mathcal{B}$ is $(j,i)$-unstable.

Note that $\mathcal{B} \neq S_{ji}(\mathcal{B})$.
The only sets of $\mathcal{B}$ which change under $S_{ji}$ are those in~$\mathcal{G}'$.
Thus, $\mathcal{G}' \neq \emptyset$, so $\mathcal{G}\neq \mathcal{A}_{i\bar{j}}$.

By Lemma \ref{lem:UW-partition},
$\mathcal{B} = \mathcal{B}_{\bar{j}i} \cup \mathcal{B}_{j\bar{i}} \cup \mathcal{B}_{\bar{j}\bar{i}}$,
where
\begin{align*}
  \mathcal{B}_{\bar{j}i}
  &= \{B\in\mathcal{B} \,:\, i\in B,\, j\notin B,\, B' = (B - \{i\}) \cup \{j\}\notin\mathcal{F}\}\\
  &= \{B\in\mathcal{B} \,:\, S_{ji}(B)\neq B\}.
\end{align*}
Comparing this expression for $\mathcal{B}$ with the definition
$\mathcal{B} = \mathcal{A}_{\bar{i}j}\cup \mathcal{A}_{\bar{i}\bar{j}}\cup (\mathcal{A}_{i\bar{j}} - \mathcal{G}) \cup \mathcal{G}'$
yields the identities
\[
  \mathcal{B}_{\bar{j}i} = \mathcal{A}_{i\bar{j}} - \mathcal{G}\,,\qquad
  \mathcal{B}_{j\bar{i}} = \mathcal{A}_{\bar{i}j}\cup \mathcal{G}'\qquad\text{and}\qquad
  \mathcal{B}_{\bar{j}\bar{i}} = \mathcal{A}_{\bar{i}\bar{j}}\,.
\]
Define
\[
  \mathcal{C} = \mathcal{B}_{\bar{j}i}       \cup
       \mathcal{B}_{\bar{j}\bar{i}} \cup
       (\mathcal{B}_{j\bar{i}} - \mathcal{H}) \cup
       \mathcal{H}'\,.
\]
We wish to show that $\mathcal{A} = \mathcal{C}$ and do this by showing
that $\mathcal{H} = \mathcal{G}'$ and $\mathcal{H}' = \mathcal{G}$.

The elements $B$ of $\mathcal{G}'$ are the sets of the form $B = (A - \{i\}) \cup \{j\}$
for some set
$A\in \mathcal{G} \subseteq \mathcal{A}_{i\bar{j}}$.
Since $(B - \{j\}) \cup \{i\} = A \in \mathcal{F}$,
it follows that $S_{ij}(B) = B$.
Thus, $\mathcal{H} = \mathcal{G}'$ and hence
\[
  \mathcal{H}' = \{(B - \{j\}) \cup \{i\} \,:\, B \in \mathcal{G'}\} = \mathcal{G}\,.
\]
Now by Lemma \ref{lem:UW-partition},
every $A \in \mathcal{A}_{\bar{i}j}$ satisfies $S_{ij}(A) \neq A$,
so   $\mathcal{H}  \cap \mathcal{A}_{\bar{i}j} = \emptyset$ and
thus $\mathcal{G}' \cap \mathcal{A}_{\bar{i}j} = \emptyset$.
Hence by Lemma \ref{lem:UW-partition},
\begin{align*}
  \mathcal{C} &= \mathcal{B}_{\bar{j}i}       \cup
        \mathcal{B}_{\bar{j}\bar{i}} \cup
        (\mathcal{B}_{j\bar{i}} - \mathcal{H}) \cup \mathcal{H}'\\
     &= \bigl((\mathcal{A}_{i\bar{j}} - \mathcal{G})  \cup \mathcal{H}'\bigr) \cup
        \mathcal{A}_{\bar{j}\bar{i}} \cup
        \bigl((\mathcal{A}_{\bar{i}j}\cup \mathcal{G}') - \mathcal{H}\bigr)\\
     &= \bigl((\mathcal{A}_{i\bar{j}} - \mathcal{G}) \cup \mathcal{G}\bigr) \cup
        \mathcal{A}_{\bar{j}\bar{i}} \cup
        \bigl((\mathcal{A}_{\bar{i}j}\cup \mathcal{G}') - \mathcal{G}'\bigr)\\
     &= \mathcal{A}_{i\bar{j}} \cup
        \mathcal{A}_{\bar{j}\bar{i}} \cup
        \mathcal{A}_{\bar{i}j}\\
     &= \mathcal{A}\,.
\end{align*}
\end{proof}

By Proposition~\ref{prop:shift_prop},
the size of a $(d,2k)$-conditionally intersecting family $\mathcal{F}\subseteq\binom{[n]}{k}$ is preserved under shifting.
Iterative shifts to $\mathcal{F}$ will result in a family $\mathcal{F}^\prime$
that is either stable or is $(i,j)$-unstable for some $i,j$;
that is, a shift applied to $\mathcal{F}^\prime$
will result in a family that is not $(d,2k)$-conditionally intersecting.

An upper bound on the size of the former is determined by Theorem~\ref{thm:Main}.
Thus to determine the maximum possible size of $(d,2k)$-conditionally intersecting families,
it suffices to consider families that are $(i,j)$-unstable for some $i,j$.

\section{Non-intersecting \texorpdfstring{$(d,2k)$}{}-conditionally intersecting families}
\label{sec:nonintersectingintersecting}

$ $\\
Theorem~\ref{thm:MR0398842} by Frankl~\cite{MR0398842} provides
an upper bound on the size of $d$-wise intersecting families.
Muyabi's Conjecture, if true, would improve this result
by providing an upper bound on size of families
that are $(d,2k)$-conditionally intersecting.
The two bounds - and the bound in Theorem~\ref{thm:Main} -
turn out to be identical,
and this bound is indeed in each case achieved by star families.

To further sharpen Mubayi's Conjecture,
it may be useful to distinguish more explicitly between the two conditions given
by Frankl's theorem and Mubayi's Conjecture.
In particular, it is worth considering the upper bounds on the size of
families addressed by the former (i.e., $d$-wise intersecting families) and the
families addressed by the latter (i.e., $(d,2k)$-conditionally intersecting families).
Hence in this section, we consider families $\mathcal{F}\in\binom{[n]}{k}$ that are $(d,2k)$-conditionally intersecting
but which are {\em not} $d$-wise intersecting;
that is, for all distinct sets $A_1,\ldots,A_d \in\mathcal{F}$
\begin{align*}
   A_1\cap\cdots\cap A_d &\neq \emptyset
   \;\;\textrm{whenever}\;\;
  |A_1\cup \cdots\cup A_d| \leq 2k   \\
  \text{but }
   A_1\cap\cdots\cap A_d &=    \emptyset
  \;\;\textrm{for at least some $d$ distinct sets}\;\:
   A_1,\ldots,A_d\in\mathcal{F}\,.
\end{align*}
In particular,
we pose the following conjecture.
\begin{conjecture}
\label{conj:nonintersecting}
For $k\geq d \geq 3$ and sufficiently large $n$,
each family $\mathcal{F}$ that is $(d,2k)$-conditionally intersecting
but which is not intersecting
has size at most
\[
  \binom{n-k-1}{k-1} + 1\,.
\]
Furthermore, equality holds if and only if
\[
  \mathcal{F} = \left\lbrace A \in \binom{[n]}{k} \,:\, x \in A \textrm{ and } A \cap B =\emptyset\right\rbrace \cup \{ B\}
\]
for some fixed $x \in [n]$ and $B \in \binom{[n]}{k}$ such that $x \notin B$.
\end{conjecture}

Note that the definition of the family $\mathcal{F}$ in the conjecture above
differs from the maximally-sized intersecting non-star family of the Hilton-Milner Theorem~\cite{MR0219428}
(see also~\cite{MR826944}),
in which the condition $A\cap B = \emptyset$
is replaced by the negated condition $A \cap B \neq\emptyset$.

An indication that Conjecture~\ref{conj:nonintersecting} might be true is
that the family $\mathcal{F}$ has size $\binom{n-k-1}{k-1}+1$ which, asymptotically,
converges to $\binom{n-1}{k-1}$, the size of star families on $n$ elements
which, in turn, are asymptotically the largest $(d,2k)$-conditionally intersecting families;
see~\cite{MR2591419,MR2507945}.

However, Conjecture~\ref{conj:nonintersecting} would be false if $n$ were allowed
to be small,
since $\mathcal{F}$ would then be smaller than other non-intersecting $(d,2k)$-conditionally intersecting
families.
For instance,
consider
\[
  \mathcal{G} = \{B_1, B_2\} \cup \left\{ A \in \binom{[n]}{k} \,:\, x,\, y\in A,\text{  and  }
         A \cap ([n] - (B_1 \cup B_2))\neq \emptyset \right\}
\]
where $B_1, B_2 \in \binom{[n]}{k}$ are fixed disjoint sets,
and $x \in B_1$ and  $y \in B_2$ are fixed elements.
The family $\mathcal{G}$ has size
$\binom{n-2}{k-2}-\binom{2k-2}{k-2}+2$ and is $(d,2k)$-intersecting,
since the only disjoint sets in $\mathcal{G}$ are $B_1$ and $B_2$ and
by definition any other member of $\mathcal{G}$ must contain an element
not in $B_1$ or $B_2$. For $n=2k+1$, $\mathcal{F}$ has size $k+1$
whereas $\mathcal{G}$ has size
\[
  \binom{(2k+1)-2}{k-2}-\binom{2k-2}{k-2}+2 =
  \binom{2k-1}{k-2}-\binom{2k-2}{k-2}+2    =
  \binom{2k-2}{k-3}+2\,.
\]
which is larger than $k+1$ for $k \geq 4$.

\medskip

For completeness, let us consider families of pairs; that is, the case in which $k=2$.
For this purpose, define a {\em twin 2-star} on $[n]$ to be any family $\mathcal{F} \subseteq \binom{[n]}{2}$ consisting, for distinct and fixed $x,y \in [n]$,
of $n-2$ unordered pairs $A\subseteq [n]$ of the form $\{z,z'\}$ where $z\in [n]-\{x,y\}$
and $z'\in\{x,y\}$ so that each element $z\in [n]-\{x,y\}$ appears exactly once in some set~$A\in\mathcal{F}$.
Note that twin stars generalise $\mathcal{F}$ in the conjecture for $k=2$.

\begin{proposition}
\label{prop:twinstar}
If $\mathcal{F} \subseteq \binom{[n]}{2}$ is $(3,4)$-conditionally intersecting
but is not intersecting, then
\[
  |\mathcal{F}| \leq n-2\,,
\]
and equality holds if and only if $\mathcal{F}$ is a twin 2-star.
\end{proposition}
\begin{proof}
Suppose that $\mathcal{F}$ is $(3,4)$-conditionally intersecting but not intersecting;
then for all distinct sets $ A_1,A_2,A_3 \in\mathcal{F}$
\begin{align}
\label{eqn:(3,4)-conditionally intersecting}
   A_1\cap A_2 \cap A_3 &\neq \emptyset
   \;\;\;\textrm{whenever}\;\;\;
  |A_1\cup A_2 \cup A_3| \leq 4\,.
\end{align}
Consider $\mathcal{F}$ to be the graph on $n$ vertices whose edges are the members of~$\mathcal{F}$.
Then (\ref{eqn:(3,4)-conditionally intersecting})
asserts that $\mathcal{F}$ contains no triangles or paths of length~3.
Therefore,
every connected component of $\mathcal{F}$ is a star.
If there are $m$ connected components in $\mathcal{F}$,
then there are $n-m$ edges in $\mathcal{F}$.
Since $\mathcal{F}$ is not intersecting,
$\mathcal{F}$ cannot be a star, so $m \geq 2$.
Hence, $|\mathcal{F}| \leq n-2$ and equality holds only if
$\mathcal{F}$ is a twin $2$-star.
\end{proof}

\section{On \texorpdfstring{$(d,s)$}{}-conditionally intersecting families for \texorpdfstring{$s < 2k$}{} }
\label{sec:families}
$ $\\
\noindent
In this section,
we review and extend some of the work by Frankl and F\"{u}redi~\cite{MR729787}
on $(d,s)$-conditionally intersecting families for $s < 2k$.
Partition $[n]$ into $k$ almost equal parts
$X_1,\ldots,X_k$ so that
$\bigl\lfloor\frac{n}{k}\bigr\rfloor \leq |X_i| \leq \bigl\lfloor \frac{n}{k}\bigr\rfloor + 1$,
and define
\[
  \mathcal{H}_k = \left\{ A \in \binom{[n]}{k} \,:\, |A \cap X_i| = 1 \textrm{ for all } i= 1,\ldots,k\right\}\,.
\]
Frankl and F\"{u}redi~\cite{MR729787} proved that
this family is $(3,2k-1)$-conditionally intersecting
and of order $\Theta(n^k)$.
Thus for fixed $k$, $\mathcal{H}_k$ is larger than a star for sufficiently large $n$.
This shows that Mubayi's Conjecture
cannot be extended to hold for $(d,s)$-conditionally intersecting families for $s < 2k$.
Motivated by these facts,
Frankl and F\"uredi~\cite{MR729787}
conjectured that $\mathcal{H}_k$ is, up to isomorphism,
uniquely largest among $(3,2k-1)$-conditionally intersecting families on $[n]$,
at least for sufficiently large $n$.
We extend this conjecture as follows.

\begin{conjecture}
\label{Conj Hr}
If $\mathcal{F} \subseteq \binom{[n]}{k}$ is $(d,2k-1)$-conditionally intersecting,
then, for sufficiently large $n$,
\[
  |\mathcal{F}| \leq |\mathcal{H}_k| \,.
\]
Furthermore,
equality holds if and only if $\mathcal{F} = \mathcal{H}_k$ up to isomorphism.
\end{conjecture}

The family $\mathcal{H}_k$ can be generalised in the following way.
Partition $[n]$ into $r \leq~k$ fixed parts (not necessarily near-equal) $X_1,\ldots,X_r$.
Choose fixed non-negative integers $x_1,\ldots,x_r$ such that
$\sum_{i=1}^{r}x_i \leq k$
and define
\begin{equation*}\label{General form}
  \mathcal{G}_r = \left\{ A \in \binom{[n]}{k} \,:\, |A \cap X_i| \geq x_i \textrm{ for all } i= 1,\ldots,r\right\}\,.
\end{equation*}

Several maximal families from Extremal Set Theory
can be expressed in this~form.

\begin{example}\label{exa:maxfamily_star}{\rm
Stars are the unique maximum families
for the Erd\H{o}s-Ko-Rado Theorem for $n > 2k$ and for Theorem \ref{thm:MR0398842}.
Each star can be expressed as follows for some $y\in [n]$:
\[
    \left\{ A \in \binom{[n]}{k} \,:\, y\in F \right\}
  = \left\{ A \in \binom{[n]}{k} \,:\, |A \cap X_1| \geq x_1 \;\;\text{and}\;\; |A \cap X_2| \geq x_2\right\}\,
\]
where $X_1 = \{y\}$ and $X_2 = [n]-\{y\}$ partition~$[n]$
and where
$x_1 = 1$ and $x_2 = 0$.}
\end{example}

\begin{example}\label{exa:maxfamily_Fi}{\rm
Define
$X_1 = [t+2j]$ and
$X_2 = [n]-X_1$ and
$x_1 = t+j$ and
$x_2 = 0$ for some integers $t\geq 1$ and $j\geq 0$.
The families
\[
    \mathcal{F}_j
  = \left\{ A \in \binom{[n]}{k}\! : |A \cap [t+2j]| \geq t+j \right\}
  = \left\{ A \in \binom{[n]}{k}\! : |A \cap X_i| \geq x_i\;\text{for}\; i = 1,2\right\}
\]
have be shown in~\cite{MR1429238} to be the largest
$t$-intersecting families $\mathcal{F} \subseteq \binom{[n]}{k}$.}
\end{example}

\begin{example}\label{exa:maxfamily_Fprime}{\rm
Define
$X_1 = [t]$,
$X_2 = [k+1]-[t]$,
$X_3 = [n]-X_1-X_2$,
$x_1 = t$,
$x_2 = 1$ and
$x_3 = 0$
for some integers $k,t\geq 1$.
The family $\mathcal{F}_1$ from Example~\ref{exa:maxfamily_Fi} and
the family
\begin{align*}
     \mathcal{F}^\prime
  &= \left\{ A \in \binom{[n]}{k} \,:\, [t] \subseteq A ,\, \bigl([k+1]-[t]\bigr) \cap A \neq \emptyset \right\}
\cup \bigl\{ [k+1]-\{ i \} \,:\, i \in [k+1]\bigr\}\\
  &= \left\{ A \in \binom{[n]}{k} \,:\, |A \cap X_i| \geq x_i \textrm{ for all } i = 1,2,3\right\}
\cup \bigl\{ [k+1]-\{ i \} \,:\, i \in [k+1]\bigr\}
\end{align*}
have been shown in \cite{MR1405994} to be, up to isomorphism,
the maximum $t$-intersecting families that satisfy
$\left|\bigcap_{A \in \mathcal{F}}A\right|<~t$.}
\end{example}

\begin{lemma}\label{Min int size}
If $n > k \geq 2$,
   $d \geq 2$ and
   $\mathcal{F} = \left\lbrace A \in 2^{[n]} \,:\,  |A| \geq k \right\rbrace$
contains $d$ sets $A_1, \ldots, A_{d}$ such that
$A_1 \cap \cdots \cap A_{d} = \emptyset$, then
\[
  n \geq \Bigl\lceil\frac{dk}{d-1}\Bigr\rceil \,.
\]
\end{lemma}

\begin{proof}
If $A_1,\ldots, A_d \in \mathcal{F}$ and $A_1 \cap \cdots \cap A_{d} = \emptyset$,
then by removing elements from each set $A_i$ of size greater than~$k$,
we obtain $d$ sets, $A^\prime_1, \ldots, A_{d}^\prime \in \binom{[n]}{k}$ ,
such that $A^\prime_1 \cap \cdots \cap A_{d}^\prime = \emptyset$.
Hence,
\[
       n(d-1)
  \geq \sum_{a\in[n]}|\{i\in[d]\,:\, a\in A^\prime_i\}|
    =  \sum_{i\in[d]}|A^\prime_i| \geq dk\,,
\]
so $n \geq \bigl\lceil \frac{dk}{d-1}\bigr\rceil$.
\end{proof}

We now determine the values of $d$ and $s$ for which
$\mathcal{G}_r$ is a $(d,s)$-conditionally intersecting family.
If $\mathcal{G}_r$ is $d$-wise intersecting,
then it is trivially $(d,s)$-conditionally intersecting for any~$s$.
The following theorem completes the determination of the values of~$d$ and~$s$.

\begin{theorem}\label{Thm Counter}
If the family $\mathcal{G}_r$ is not $d$-wise intersecting,
then it is $(d,s)$-conditionally intersecting but not
$(d,s+1)$-conditionally intersecting,
where
\[
  s = \max\biggl\{ \Bigl\lceil\frac{dk}{d-1}\Bigr\rceil -1,\sum_{i=1}^{r}\Bigl\lceil\frac{dx_i}{d-1}\Bigr\rceil -1 \biggr\}\,.
\]
\end{theorem}

\begin{proof}
Suppose that the family $\mathcal{G}_r$ is not $d$-wise intersecting,
and let $s$ be defined as above.
Let $A_1,\ldots,A_d \in \mathcal{G}_r$ be $d$ sets with empty intersection.
For each $i$, $A_j \cap X_i$ is a subset of
$(A_1 \cup\cdots\cup  A_d)\cap X_i$
of size at least~$x_i$.
Also,
\[
  (A_1 \cap X_i) \cap \cdots \cap (A_{d} \cap X_i) \subseteq A_1 \cap \cdots \cap A_d = \emptyset\,.
\]
Hence by Lemma \ref{Min int size},
\begin{equation}\label{eqn: lower bound in Xi}
  |\left(A_1 \cup\cdots\cup  A_d \right)\cap X_i| \geq \Bigl\lceil \frac{dx_i}{d-1}\Bigr\rceil\,.
\end{equation}
The sets $X_1,\ldots,X_r$ partition $[n]$,
so
\[
      | A_1\cup\cdots\cup  A_d|
  =   \sum_{i=1}^{r} |(A_1 \cup\cdots\cup  A_d) \cap X_i|
 \geq \sum_{i=1}^{r} \Bigl\lceil \frac{dx_i}{d-1}\Bigr\rceil\,.
\]
Therefore,
$\mathcal{G}_r$ is
$\Bigl(d, \sum_{i=1}^{r} \bigl\lceil \frac{dx_i}{d-1}\bigr\rceil-1 \Bigr)$-conditionally intersecting.
Also
by Lemma \ref{Min int size},
\[
  |A_1 \cup \cdots \cup A_d| \geq  \Bigl\lceil \frac{dk}{d-1}\Bigr\rceil\,.
\]
Thus,
$\mathcal{G}_r$ is  $\Bigl(d,\bigl\lceil \frac{dk}{d-1}\bigr\rceil -1 \Bigr)$-conditionally intersecting
and thus $(d,s)$-conditionally intersecting.

We now show that $\mathcal{G}_r$ is not $(d, s+1)$-conditionally intersecting
by finding $d$ sets in $\mathcal{G}_r$ with empty intersection and union of size $s+1$.
For each $i$,
define
\[
  y_i := \Bigl\lceil \frac{dx_i}{d-1}\Bigr\rceil \,.
\]
As $\mathcal{G}_r$ is not $d$-wise intersecting,
there are $d$ members $A_1,\ldots, A_d$ of $\mathcal{G}_r$ with empty intersection.
Then by (\ref{eqn: lower bound in Xi}),
$|X_{i}| \geq y_i$ for all $i$.
Choose $X_i' \subseteq X_{i}$ such that $|X_i'| = y_i$
and let $a_i$ be any integer satisfying $d x_i \leq a_i \leq  (d-1)y_i$.
We will now construct $d$ subsets $A_{i,1},\ldots,A_{i,d}\subseteq X_i'$
each of size $x_i$ or $x_i+1$ and satisfying
$A_{i,1} \cap \cdots  \cap A_{i,d} = \emptyset$
and
\[
  \sum_{j = 1}^{d}|A_{i,j}| = a_i\,.
\]
First, set $A_{i,1} = \cdots  = A_{i,d} = \emptyset$.
Add an arbitrary element $x\in X_i'$ to $d-1$ of the sets $A_{i,j}$.
Choose a new element $x' \in X_i'$
and add it to $d-1$ of the smallest of the sets $A_{i,j}$.
Continue to add new elements to $d-1$ of the smallest sets
until all but one element of $X_i'$ have been added.
Add the remaining element to $a_i - (d-1)(y_i-1)$
of the smallest of the sets $A_{i,j}$.
The resulting sets $A_{i,1},\ldots,A_{i,d}$ satisfy
 $A_{i,1} \cap \cdots  \cap A_{i,d} = \emptyset$
as no element was added to all of the sets $A_{i,j}$.
Also,
\[
    \sum_{j = 1}^{d}|A_{i,j}|
  = \sum_{x \in X_i^\prime}\bigl|\{\,j \,:\, x \in A_{i,j}\}\bigr|
  = (d-1)(y_i -1) + a_i -
    (d-1)(y_i -1)
  = a_i \,.
\]
Finally, the sets $A_{i,1},\ldots,A_{i,d}$ are of near-equal size
and as
\[
       d x_i
  \leq \sum_{j = 1}^{d}|A_{i,j}| \leq (d-1)y_i
    <  d (x_i+1)\,,
\]
each set $A_{i,j}$ has size $x_i$ or $x_{i+1}$.

Given an integer $a$ with $\sum_{i=1}^{r} d x_{i} \leq a \leq \sum_{i=1}^{r}(d-1)y_i$,
we construct new sets $A_{1},\ldots,A_{d}$ of near-equal size
that satisfy $A_{1} \cap \cdots \cap A_{d} =~\emptyset$,
$x_i \leq |A_j \cap X_i | \leq x_i+1$ for all $i$ and $j$,
and
\[
  \sum_{j = 1}^d |A_j| = a\,.
\]
Let $a_1,\ldots,a_r$ be integers such that
$d x_i \leq a_i \leq  (d-1)y_i$
and
$a_1 + \cdots + a_r = a$.
Construct $A_{i,1},\ldots,A_{i,d}$ as above for all $i$,
and set $A_j = A_{1,j}$ for all $j$.
Relabel the $j$ indices so that the indices $j$ for which $|A_{2,j}| = x_{2}+1$ are
also the indices $j$
for which $A_j$ are the smallest of the sets $A_1,\ldots, A_{d}$.
Now add the elements of $A_{2,j}$ to $A_j$ for all $j$; i.e.,
let $A_{j} = A_{1,j} \cup A_{2,j} $.
Again, relabel the sets so that $|A_{3,j}| = x_{3}+1$
for the indices $j$ for which the sets $A_{j}$ are
the smallest of the sets $A_1, \ldots, A_{d}$.
Continue to add $A_{i,j}$ and relabel in this way, for $i = 3,\ldots,r$ and all $j$.
We have thus constructed sets
\[
  A_j = \bigcup_{i=1}^{r} A_{i,j}
\]
for all $j$.
Since $A_{i,1} \cap \cdots  \cap A_{i,d} = \emptyset$ for all $i$,
it follows that $A_{1} \cap \cdots \cap A_{d} = \emptyset$.
Note that $x_i \leq |A_j \cap X_i| = |A_{i,j}| \leq x_i + 1$ for all $i$ and $j$,
and that, by construction,
the sets $A_j$ are near-equally sized
and that
\[
    \sum_{j = 1}^d |A_j|
  = \sum_{j = 1}^d \sum_{i = 1}^r |A_{i,j}|
  = \sum_{i = 1}^r \sum_{j = 1}^d |A_{i,j}|
  = \sum_{i = 1}^r a_i
  = a\,.
\]
We now wish to prove that, for some $a$,
the sets $A_1,\ldots,A_d$ constructed in the way described above are in $\mathcal{G}_r$
and that the union of the sets $A_1,\ldots,A_d$ has size $s+1$.
For this purpose, we consider two cases,
namely the case in which $\sum_{i=1}^{r}(d-1)y_i$ is greater than $dk$,
and the cases in which it is not.

\medskip

\noindent
{\sc Case I}:
Suppose that $\sum_{i=1}^{r}(d-1)y_i \geq dk$.
Then
\[
    \sum_{i=1}^{r}y_i
  = \max\biggl\lbrace \Bigl\lceil\frac{dk  }{d-1}\Bigr\rceil,
                      \sum_{i=1}^{r} y_i\biggr\rbrace
  = \max\biggl\lbrace \Bigl\lceil\frac{dk  }{d-1}\Bigr\rceil,
                      \sum_{i=1}^{r}\Bigl\lceil\frac{dx_i}{d-1}\Bigr\rceil \biggr\rbrace
  = s + 1\,.
\]
Also,
\[
  \sum_{i=1}^{r} dx_i
  \leq dk\leq
  \sum_{i=1}^{r}(d-1)y_i\,,
\]
so we were allowed to choose the value $a = dk$ in the above construction,
and thus obtain the sets $A_1, \ldots, A_d$ of near-equal size
satisfying
\[
  \sum_{i = 1}^{d}|A_{i}| = a = dk\,.
\]
Since $|A_j|\leq k$ for each $j$,
it follows that $|A_1| = \cdots = |A_d| = k$.
Since $|A_j \cap X_i|\geq x_i$ for all $i$ and $j$,
it follows that $A_1,\ldots, A_{d} \in \mathcal{G}_r$.
By construction,
\[
    |A_1 \cup \ldots \cup A_{d}|
  = \sum_{i=1}^r |X_i^\prime|
  = \sum_{i=1}^{r}y_i
  =  s+1\,.
\]
\noindent
{\sc Case II}:
Now suppose that $\sum_{i=1}^{r}(d-1)y_i < dk$.
Then
\[
    s
  = \max\biggl\lbrace \Bigl\lceil\frac{dk  }{d-1}\Bigr\rceil - 1,
        \sum_{i=1}^{r}\Bigl\lceil\frac{dx_i}{d-1}\Bigr\rceil - 1 \biggr\rbrace
  = \max\biggl\lbrace \Bigl\lceil\frac{dk  }{d-1}\Bigr\rceil - 1,
        \sum_{i=1}^{r} y_i - 1 \biggr\rbrace
  = \Bigl\lceil \frac{dk}{d-1}\Bigr\rceil-1\,.
\]
Also, we may assume that the value $a = \sum_{i=1}^{r}(d-1)y_i$ was used in the above construction,
so the resulting sets $A_1,\ldots, A_d$ satisfy
\[
  \sum_{i = 1}^{d}|A_{i}| = a = \sum_{i=1}^{r}(d-1)y_i\,.
\]
Since $a_i \leq (d-1)y_i$ for each $i$ and $a_1 + \cdots a_r = a = \sum_{i=1}^{r}(d-1)y_i$,
it follows that $a_i = (d-1)y_i$ for each $i$.
Hence in the construction,
each element in $X' := X_1^\prime \cup \cdots \cup X_{r}^\prime$
was added to exactly $d-1$ of the sets $A_1,\ldots, A_d$
or, in the case of the $r$ final elements,
to
\[
  a_i - (d - 1)(y_i - 1) = (d-1)y_i  - (d - 1)(y_i - 1) = d-1
\]
of the sets.
That is, each element of $X'$ lies in exactly $d-1$ of the sets $A_1,\ldots, A_d$.
Add an element of $[n]- X'$
to $d-1$ of the smallest sets $A_i$
and continue to add elements of $[n] - X'$
in this way until
\[
  dk \leq  \sum_{i = 1}^{d}|A_{i}| < dk + d-1\,.
\]
This is possible, for suppose by way of contradiction that it were not;
then every element of~$[n]$ would lie in exactly $d-1$ of the sets $A_1,\ldots, A_d$
and yet $\sum_{i = 1}^{d}|A_{i}| < dk$.
Then
\[
    (d-1)n
  = \sum_{j \in [n]} |\{A_i \,:\, j \in A_i \}|
  = \sum_{i = 1}^{d}|A_{i}|< dk \,.
\]
However, $\mathcal{G}_r$ is not $d$-wise
intersecting, so $n \geq \bigl\lceil \frac{dk}{d-1}\bigr\rceil$ by Lemma~\ref{Min int size},
a contradiction.

Now remove an element from a largest set $A_i$
and continue to do so until
\[
  \sum_{i = 1}^{d}|A_{i}| = dk\,.
\]
As in {\sc Case I},
the sets $A_1, \ldots, A_d$ have near-equal size and $|A_j|\leq k$ for each $j$,
so $|A_1| = \cdots = |A_d| = k$.
As $|A_j \cap X_i| \geq x_j$ for all $i$ and $j$,
it follows that $A_1, \ldots, A_d \in \mathcal{G}_r$.
The number of elements of
$[n]- X'$ added to the sets $A_1,\ldots, A_d $ is
\[
    \biggl\lceil \frac{dk - \sum_{i=1}^{r}(d-1)y_i}{d-1} \biggr\rceil
  = \Bigl\lceil \frac{dk}{d-1}\Bigr\rceil -  \sum_{i=1}^{r}y_i
\]
and the number of elements of $X' = X_1^\prime \cup \cdots \cup X_{r}^\prime$
added to the sets $A_1, \ldots, A_d $ is $\sum_{i=1}^{r}y_i$,
so
\[
    |A_1 \cup \cdots \cup A_d|
  = \Bigl\lceil \frac{dk}{d-1}\Bigr\rceil
  - \sum_{i=1}^{r}\Bigl\lceil \frac{dk}{d-1}\Bigr\rceil
  + \sum_{i=1}^{r}y_i
  = \Bigl\lceil \frac{dk}{d-1}\Bigr\rceil = s+1\,.
\]
In both {\sc Case I} and {\sc Case II},
we have $d$ sets $A_1,\ldots,A_d\in\mathcal{G}_r$
with union size $s+1$ and yet empty mutual intersection.
Hence, the sets $A_1,\ldots,A_{d}$ lie in $\mathcal{G}_r$ and are not $(d,s+1)$-conditionally intersecting,
so neither is $\mathcal{G}_r$.
\end{proof}

\begin{example}\label{exa:maxfamily_Fi2}{\rm
By Theorem~\ref{Thm Counter}, the family $\mathcal{F}_j$ from Example~\ref{exa:maxfamily_Fi} is
$\bigl(d,\bigl\lceil\frac{dk}{d-1}\bigr\rceil -1 \bigr)$-conditionally intersecting
for all $d \geq 3$
but is not
$\bigl(d,\bigl\lceil\frac{dk}{d-1}\bigr\rceil    \bigr)$-conditionally intersecting.}
\end{example}

\section{On \texorpdfstring{$(2,s)$}{}-conditionally intersecting families of \texorpdfstring{$2^{[n]}$}{}: Part I}
\label{sec:2si}
$ $\\
\noindent
Erd\H{o}s et al.~\cite{MR0140419} proved the following bound
on the size of each intersecting family $\mathcal{F} \subseteq 2^{[n]}$.

\begin{lemma}{\rm\cite{MR0140419}}\label{Lem int}
If $\mathcal{F} \subseteq 2^{[n]}$ is intersecting, then
\[
  |\mathcal{F}| \leq 2^{n-1}\,.
\]
\end{lemma}

A star meets the bound above
but is not the unique maximum family.
While weak asymptotic results are known (see~\cite{MR3066347}),
it is a difficult challenge to classify the case of equality for general~$n$.
It is also difficult to count the number of families attaining the bound of Lemma~\ref{Lem int},
even for small~$n$.
The highest value of $n$ for which such counting has presently been achieved is~9;
see~\cite{MR3066347}.

The aim of this section and that of Section~\ref{sec:2sii} is to extend Lemma~\ref{Lem int}.
To this aim,
we extend the $(2,s)$-conditionally intersecting condition on families
$\mathcal{F} \subseteq \binom{[n]}{k}$ to families $\mathcal{F} \subseteq 2^{[n]}$.
In particular,
a family $\mathcal{F} \subseteq 2^{[n]}$ is {\em $(2,s)$-conditionally intersecting}
if
\[
   A \cap B \neq \emptyset
   \quad\textrm{whenever}\quad
  |A \cup B| \leq s
   \quad\textrm{for each }
   A, B \in \mathcal{F} \,.
\]
Theorem~\ref{thm:mine} below provides tight upper bounds on the sizes
of $(2,s)$-conditionally intersecting families $\mathcal{F} \subseteq 2^{[n]}$,
and describing precisely the families that attain these bounds.
This theorem extends Lemma~\ref{Lem int};
the latter is obtained from the former by setting $s = n$.

\begin{theorem}\label{thm:mine}
Suppose that $\mathcal{F}\subseteq 2^{[n]}$ is $(2,s)$-conditionally intersecting for some $s\leq n$.
If $s=2k$, then
\[
  |\mathcal{F}|\leq \binom{n-1}{k-1}+\sum_{i=k+1}^{n}\binom{n}{i}
\]
and if $s < n$ also holds, then equality holds above if and only if, up to isomorphism,
\[
  \mathcal{F} = \Bigl\{ F\in \binom{[n]}{k} \,:\, 1\in F\Bigr\}
       \cup
       \binom{[n]}{>k}\,.
\]
If $s=2k-1$, then
\[
  |\mathcal{F}|\leq \sum_{i=k}^{n}\binom{n}{i}
\]
and if $s < n$ also holds, then equality holds above if and only if $\mathcal{F} = \binom{[n]}{\geq k}$.
\end{theorem}

\noindent
Here, $\binom{[n]}{>k}$ and $\binom{[n]}{\geq k}$
denote the respective families
\[
  \binom{[n]}{>k} = \{F\in 2^{[n]} \,:\, |F|>k\}
  \quad\text{and}\quad
  \binom{[n]}{\geq k} = \{F\in 2^{[n]} \,:\, |F|\geq k\}\,.
\]
Note that Keevash and Mubayi~\cite{MR2676834}
also presented a conditionally intersecting definition with the same terminology as ours;
their definition is however different from that given here,
and their results are only tangentially related to the results of this section.

In order to prove Theorem~\ref{thm:mine},
we require some preliminary definitions and results.
For a family $\mathcal{F} \subseteq \binom{[n]}{k}$,
define
\[
  \mathcal{F}^C = \{ F^C      \,:\,  F \in \mathcal{F} \}\qquad\text{and}\qquad
  \mathcal{F}_r = \{ F \in \mathcal{F} \,:\, |F| = r \}\,.
\]
Also, for each $\ell \leq k$,
define the {\em $\ell$-shadow} of $\mathcal{F}$ to be the set
\[
  \sigma_\ell(\mathcal{F})
  = \bigl\{ F \in \binom{[n]}{\ell} \,:\,
            F \subseteq A \textrm{ for some } A \in \mathcal{F} \bigr\} \,.
\]

The following theorem by Katona~\cite{MR0168468}
gives a lower bound for the size of a shadow under certain conditions.
\begin{theorem}\label{Katona}{\rm\cite{MR0168468}}
If $0 \leq k-t \leq \ell \leq k$
and $\mathcal{F} \subseteq \binom{[n]}{k}$ is $t$-intersecting,
i.e., $|A\cap B| \geq t$ for all $A,B \in \mathcal{F}$,
then
\[
  |\sigma_\ell(\mathcal{F})| \geq \frac{\binom{2k-t}{\ell}}{\binom{2k-t}{k}}|\mathcal{F}|\,.
\]
Furthermore, equality holds if and only if $n = 2k-t$, $ \mathcal{F} = \binom{[n]}{k}$
and $\ell\in\{k, k-t\}$.
\end{theorem}

\noindent
We are now ready to prove Theorem~\ref{thm:mine}
and will follow an approach similar to that in~\cite{MR0168468}.

\begin{proof}[Proof of Theorem \ref{thm:mine}]
First suppose that $\mathcal{F}$ contains no set of size less than $k =~\lceil\frac{s}{2}\rceil$.
Suppose that $s = 2k$ and define
\[
    \mathcal{F}^+
  = \mathcal{F}   \cup \binom{[n]}{>k}
  = \mathcal{F}_k \cup \binom{[n]}{>k}\,.
\]
If $A,B\in\mathcal{F}_k$ are distinct sets in $\mathcal{F}$,
then $|A\cup B|\leq 2k = s$,
so $A\cap B\neq\emptyset$
since $\mathcal{F}$ is $(2,s)$-conditionally intersecting.
Hence, $\mathcal{F}_k$ is an intersecting family,
so by the Erd\H{o}s-Ko-Rado~Theorem~\cite{MR0140419},
\[
       |\mathcal{F}|
  \leq |\mathcal{F}^+|
    =  |\mathcal{F}_k| + \biggl|\binom{[n]}{>k}\biggr|
  \leq \binom{n-1}{k-1}+\sum_{i=k+1}^{n}\binom{n}{i}\,,
\]
and the unique maximum family (up to isomorphism) is
\[
  \mathcal{F}   =  \Bigl\{F\in \binom{[n]}{k} \,:\, 1\in F\Bigr\} \cup \binom{[n]}{>k}\,.
\]
If $s = 2k-1$,
then $\mathcal{F}\subseteq \binom{[n]}{\geq k}$, so
\[
  |\mathcal{F}| \leq \sum_{i=k}^{n}\binom{n}{i}\,,
\]
and the unique maximum family is $\binom{[n]}{\geq k}$.

Now suppose that $\mathcal{F}_{r'}\neq \emptyset$ for some positive integer $r'<k$
and define
\[
  \mathcal{F}' = \biggl(\mathcal{F} \cup \bigcup_{r=1}^{k-1} \sigma_{s-r}\bigl((\mathcal{F}_r)^C\bigr)\biggr) - \bigcup_{r=1}^{k-1} \mathcal{F}_r\,.
\]
We first show that $|\mathcal{F}| \leq |\mathcal{F}'|$.
If $A\in \sigma_{s-r}\bigl((\mathcal{F}_r)^C\bigr)$,
then
$A\subseteq F^C$ for some $F\in\mathcal{F}_r$,
so
\[
  F\cap A = \emptyset \quad \textrm{and} \quad |A\cup F| = |A| + |F|= (s-r) + r = s\,.
\]
Hence, $A\notin\mathcal{F}$ since $F\in\mathcal{F}$ and $\mathcal{F}$ is $(2,s)$-conditionally intersecting;
therefore,
$\mathcal{F}\,\cap \, \sigma_{s-r}\bigl((\mathcal{F}_r)^C\bigr) = \emptyset$.
For any $A,B \in (\mathcal{F}_r)^C$,
\[
  |A \cap B| = |A|+|B|-|A \cup B| \geq 2(n-r) - n = n-2r\,,
\]
so $(\mathcal{F}_r)^C$ is $(n-2r)$-intersecting.
As $s-r \geq r = (n-r)-(n-2r)$,
Theorem~\ref{Katona} implies that
\[
       |\sigma_{s-r}\bigl((\mathcal{F}_r)^C\bigr)|
  \geq \frac{\binom{n}{s-r}}{\binom{n}{n-r}}|(\mathcal{F}_r)^C|
   =   \frac{\binom{n}{s-r}}{\binom{n}{n-r}} |\mathcal{F}_r|\,.
\]
Now, $s-r\geq r$, so if $s-r\leq \frac{n}{2}$,
then $\binom{n}{s-r}\geq \binom{n}{r} = \binom{n}{n-r}$.
Otherwise, $\frac{n}{2}\leq s-r\leq n-r$;
also in this case,
$\binom{n}{s-r} \geq \binom{n}{n-r}$.
Hence, $|\sigma_{s-r}\bigl((\mathcal{F}_r)^C\bigr)| \geq |\mathcal{F}_r|$,
so
\[
     |\mathcal{F}'|
   = |\mathcal{F}| + \sum_{r=1}^{k-1} \bigl|\sigma_{s-r}\bigl((\mathcal{F}_r)^C\bigr)\bigr| - \sum_{r=1}^{k-1} |\mathcal{F}_r|
   =  |\mathcal{F}| + \sum_{r=1}^{k-1} \bigl(\bigl|\sigma_{s-r}\bigl((\mathcal{F}_r)^C\bigr)\bigr| - |\mathcal{F}_r|\bigr)
 \geq |\mathcal{F}|\,.
\]
Next, we prove that $\mathcal{F}'$ is $(2,s)$-conditionally intersecting.
Suppose that $|A\cup B| \leq s$ for $A,B\in\mathcal{F}'$.
If $A,B\in\mathcal{F}$, then $A\cap B\neq \emptyset$ since $\mathcal{F}$ is $(2,s)$-conditionally intersecting;
otherwise, either $A$ or $B$ is an element of $\mathcal{F}'-\mathcal{F} = \bigcup_{r=1}^{k-1} \sigma_{s-r}\bigl((\mathcal{F}_r)^C\bigr)$,
without loss of generality $A \in \sigma_{s-r}\bigl((\mathcal{F}_r)^C\bigr)$.
Then
\[
       |A\cap B|
    =  |A| + |B| - |A\cup B|
  \geq (s-r) +  k  - s
    = k-r
  > 0\,,
\]
so $A\cap B\neq \emptyset$.
We conclude that $\mathcal{F}'$ is $(2,s)$-conditionally intersecting.

Since $\mathcal{F}'$ contains no set of size smaller than $k$,
the first part of the proof implies that
\[
  |\mathcal{F}|\leq|\mathcal{F}'| \leq \sum_{i=k}^n\binom{n}{i}
\]
when $s=2k-1$ and that, when $s=2k$,
\[
  |\mathcal{F}|\leq|\mathcal{F}'| \leq \binom{n-1}{k-1}+\sum_{i=k+1}^{n}\binom{n}{i}\,.
\]
Finally, assume that bounds above met with $s<n$ and that $\mathcal{F}_r\neq\emptyset$ for some $r < k$.
By the proof above,
it follows that
\[
    |\sigma_{s-r}\bigl((\mathcal{F}_r)^C\bigr)|
  = \frac{\binom{n}{s-r}}{\binom{n}{n-r}}\bigl|(\mathcal{F}_r)^C\bigr|\,,
\]
so Theorem~\ref{Katona} implies that
$\mathcal{F}_r^C = \binom{[n]}{n-r}$ and so $\mathcal{F}_r = \binom{[n]}{r}$.
However, $n > 2r$, so $\mathcal{F}_r = \binom{[n]}{r}$ is not intersecting,
so there are disjoint sets $A,B\in\mathcal{F}$ with $|A \cup B | = 2r \leq s$,
contradicting that $\mathcal{F}$ is $(2,s)$-conditionally intersecting.
Hence when $s<n$ and $\mathcal{F}$ contains sets of size less than~$k$,
$|\mathcal{F}|$ cannot reach its upper bound.
This concludes the proof.
\end{proof}

\section{On \texorpdfstring{$(2,s)$}{}-conditionally intersecting families of \texorpdfstring{$2^{[n]}$}{}: Part II}
\label{sec:2sii}
$ $\\
\noindent
In the previous section, we extended Lemma~\ref{Lem int} by
presenting and proving the more general Theorem~\ref{thm:mine}.
In this final section, we will extend this line of investigation further,
by expanding the focus on $2^{[n]}$ to the more general set of families
\[
  \binom{[n]}{\leq u} = \bigl\{ X\subseteq [n] \,:\, |X|\leq u\}
\]
for certain values of $u$,
as follows.

\begin{theorem}\label{thm:mine2}
Let $\mathcal{F} \subseteq \binom{[n]}{\leq u}$ be $(2,s)$-conditionally intersecting for some $s\leq n$.
\begin{itemize}
  \item[(i)] If $u \geq s - 1$ and $s = 2k$, then
  \[
    |\mathcal{F}|\leq \binom{n-1}{k-1}+\sum_{i=k+1}^u\binom{n}{i}
  \]
  and if $s < n$ also holds, then equality holds if and only if, up to isomorphism,
  \begin{equation}\label{max form one}
    \mathcal{F} = \{ F\in 2^{[n]} \,:\, k<|F|\leq u \} \cup
      \Bigl\{ F\in \binom{[n]}{k} \,:\, 1\in F\Bigr\}\,.
    \end{equation}
  \item[(ii)] If $u \geq s - 1$ and $s = 2k - 1$, then
    \[
      |\mathcal{F}|\leq \sum_{i=k}^u \binom{n}{i}
    \]
    and if $s < n$ also holds, then equality holds above if and only if
    \begin{equation}\label{max form two}
      \mathcal{F} = \{ F\in 2^{[n]} \,:\, k\leq |F|\leq u \} \,.
    \end{equation}
  \item[(iii)] If $u \leq k =\bigl\lfloor\frac{s}{2}\bigr\rfloor$,
    then
    \[
      |\mathcal{F}| \leq \sum_{r=1}^{u}\binom{n-1}{r-1}
    \]
    and equality holds if and only if, up to isomorphism,
    \begin{equation}\label{max form three}
      \mathcal{F} = \{ F\in 2^{[n]} \,:\, k\leq |F|\leq u \text{ and } 1\in F\} \,.
    \end{equation}
  \end{itemize}
\end{theorem}

\begin{proof}
The first two parts of the theorem follow easily from the proof of Theorem~\ref{thm:mine}
and the observation that one may freely add or remove sets of size at least~$s$
without violating the $(2,s)$-conditionally intersecting condition.
Suppose therefore that $u \leq k =\bigl\lfloor \frac{s}{2}\bigr\rfloor$.
Then $\mathcal{F}$ is intersecting;
in particular, the subfamilies
\[
  \mathcal{F}_r = \{ F \in \mathcal{F} \,:\, |F| = r \}
\]
are intersecting (so $|\mathcal{F}_1|\leq 1$).
By the Erd\H{o}s-Ko-Rado~Theorem~\cite{MR0140419}, each family $\mathcal{F}_r$ has
at most $\binom{n-1}{r-1}$ members,
and this bound is achieved if and only if $\mathcal{F}_r$ is a star.
Hence, $|\mathcal{F}|\leq\sum_{r=1}^u\binom{n-1}{r-1}$,
and this bound is achieved if and only if $\mathcal{F}$ is the union of stars $\mathcal{F}_1,\ldots,\mathcal{F}_u$
whose members, since $\mathcal{F}$ is intersecting, must each contain some fixed element,
namely that in the single member of $\mathcal{F}_1$.
\end{proof}

Considering the results of Theorem~\ref{thm:mine2} leads naturally to the following question.

\begin{question}\label{que:starornotstarthatisthequestion}
For each given value~$u$,
must each $(2,s)$-conditionally intersecting family $\mathcal{F}\subseteq \binom{[n]}{\leq u}$
of maximal size for such families
necessarily be of one of the forms (\ref{max form one})--(\ref{max form three})?
\end{question}

In general, this appears to be a difficult question.
We can however answer it positively for sufficiently large values $n$
with the next - and last - result of this paper, Theorem~\ref{thm:Family-size-comp}.

\begin{theorem}\label{thm:Family-size-comp}
Let $\mathcal{F} \subseteq \binom{[n]}{\leq u}$ be $(2,s)$-conditionally intersecting
for $\frac{s}{2} < u < s-1$.
If $s=2k$, then
\[
  |\mathcal{F}| \leq \binom{n-1}{k-1}+\sum_{i=k+1}^u \binom{n}{i} \,,
\]
and, for $n$ sufficiently larger than $s-r$,
equality holds if and only if, up to isomorphism,
\[
 \mathcal{F} = \Bigl\{ F\in \binom{[n]}{k} \,:\, 1\in F\Bigr\}
      \cup \{F\in 2^{[n]} \,:\, k<|F|\leq u\}\,.
\]
If $s=2k-1$, then
\[
  |\mathcal{F}| \leq \sum_{i=k}^u \binom{n}{i}\,,
\]
and, for $n$ sufficiently larger than $s-r$,
equality holds if and only if, up to isomorphism,
\[
  \mathcal{F} = \{F\in 2^{[n]} \,:\, k\leq |F|\leq u\}\,.
\]
\end{theorem}

\begin{proof}
As in the proof of Theorem~\ref{thm:mine},
the bounds of Theorem~\ref{thm:Family-size-comp} hold
when $\mathcal{F}$ contains no set of size less than $k = \lceil\frac{s}{2}\rceil$,
as do the characterisations of the extremal families.

Suppose then that $\mathcal{F}_r\neq \emptyset$ for some positive integer $r<k$.
To conclude the proof, we will prove that $\mathcal{F}$ cannot achieve
the bounds of the theorem for $n$ sufficiently larger than~$s$.
By adding sets to $\mathcal{F}$,
we may suppose that $\mathcal{F}$ is maximal among $(2,s)$-conditionally intersecting subfamilies~of~$\binom{[n]}{\leq u}$
that contain some member of size strictly less than $k$.
Define
\[
  \mathcal{F}' = \biggl(\mathcal{F} \cup \bigcup_{r=s-u}^{k-1} \sigma_{s-r}\bigl((\mathcal{F}_r)^C\bigr)\biggr)
                     - \bigcup_{r=1}^{k-1} \mathcal{F}_r\,.
\]
By the proof of Theorem~\ref{thm:mine},
$\mathcal{F}'$ is $(2,s)$-conditionally intersecting.
Since $\mathcal{F}'$ contains no set of size smaller than $k$ or any set larger than $u$,
the first part of the proof shows that $|\mathcal{F}'|$ satisfies the bounds of the theorem.
We will prove that $|\mathcal{F}| < |\mathcal{F}'|$ for $n$ sufficiently larger than~$s$,
implying that $|\mathcal{F}|$ cannot achieve the given bounds.

As in the proof of Theorem~\ref{thm:mine},
$\mathcal{F}\cap \sigma_{s-r}\bigl((\mathcal{F}_r)^C\bigr) = \emptyset$ for each $r$,
so
\[
    |\mathcal{F}'|
  = |\mathcal{F} | + \sum_{r=s-u}^{k-1} \bigl|\sigma_{s-r}\bigl((\mathcal{F}_r)^C\bigr)\bigr| - \sum_{r=1}^{k-1} |\mathcal{F}_r|\,.
\]
Thus to prove that $|\mathcal{F}| < |\mathcal{F}'|$ for $n$ sufficiently larger than $s-r$,
we must prove that
\begin{equation}\label{eqn: required eqn}
  \sum_{r=r'}^{k-1}|\mathcal{F}_r| < \sum_{r=s-u}^{k-1}\bigl|\sigma_{s-r}\bigl((\mathcal{F}_r)^C\bigr)\bigr|
\end{equation}
for $n$ sufficiently larger than $s-r$,
where $r'$ is the smallest integer for which $\mathcal{F}_{r'}\neq\emptyset$.
By the proof of Theorem~\ref{thm:mine},
(\ref{eqn: required eqn}) is true when $r' \geq s-u$,
so suppose that $r' < s-u \,(\leq k)$.
For each integer $r$ with $s-u \leq r \leq k-1$,
Theorem~\ref{Katona} implies that
\[
       \frac{\binom{n}{s-r}}{\binom{n}{r}} |\mathcal{F}_r|
  \leq \bigl|\sigma_{s-r}\bigl((\mathcal{F}_r)^C\bigr)\bigr| \,.
\]
Furthermore for such $r$,
we have $r \leq k-1 \leq (2k-1)-k = s-k < s-r$,
so $\binom{n}{s-r}/\binom{n}{r} \to \infty$ for $n-(s-r) \rightarrow \infty$.

For each $r = s-u \,(\,> r^\prime),\ldots, k-1$
and $r'' = r', \ldots , s-u-1$,
consider any $A \in \mathcal{F}_{r''}$ and $B \supset A $ with $|B| = r$.
If $B \notin \mathcal{F}_{r}$,
then adding $B$ to $\mathcal{F}$ would create a larger
$(2,s)$-conditionally intersecting subfamily of $\binom{[n]}{\leq u}$
with at least one member of size less than~$k$, contradicting the maximality of $\mathcal{F}$.
Hence, $B \in \mathcal{F}_{r}$.
Double counting the pairs $(A,B)$ with $A \in \mathcal{F}_{r''}$ and $A \subset B \in \mathcal{F}_{r}$
yields the inequality
\[
  |\mathcal{F}_r| \binom{r}{r''} \geq \binom{n-r''}{r-r''}|\mathcal{F}_{r''}|\,.
\]
%from which the claimed inequality follows.
In particular, there is a constant $C$ such that
$\sum_{r=r'}^{s-u-1} |\mathcal{F}_r| < C|\mathcal{F}_{k-1}|$
for sufficiently large $n$.
Hence, for $n$ sufficiently larger than $s-r$,
\[
       \sum_{r=r'}^{k-1} |\mathcal{F}_r|
    <  \sum_{r=s-u}^{k-1}\frac{\binom{n}{s-r}}{\binom{n}{r}}|\mathcal{F}_{r}|
  \leq \sum_{r=s-u}^{k-1}\bigl|\sigma_{s-r}\bigl((\mathcal{F}_r)^C\bigr)\bigr|\,,
\]
and thus $|\mathcal{F}|<|\mathcal{F}'|$.
This concludes the proof.
\end{proof}

\medskip

\section*{Acknowledgements}
We thank the referees for their corrections and suggestions which improved the presentation of this paper.

\bibliographystyle{plain}
\bibliography{OnMubayisConjectureAndConditionallyIntersectingSetsFinal3}
\end{document}